\documentclass[12pt,a4paper]{article}
\usepackage{amsfonts}
\usepackage{amsmath}
\usepackage{amsthm}
\usepackage{amssymb}
\usepackage{tikz-cd}
\usepackage{tikz}
\usepackage{diagbox}
\usepackage{subcaption}
\theoremstyle{plain}
\newtheorem{theorem}{Theorem}[section]
\newtheorem{lemma}[theorem]{Lemma}
\newtheorem{definition}[theorem]{Definition}
\newtheorem{corollary}[theorem]{Corollary}
\newtheorem{remark}[theorem]{Remark}
\newtheorem{proposition}[theorem]{Proposition}
\DeclareMathOperator{\ord}{ord}
\author{Tadeusz Krasiński \& Aleksandra Zakrzewska}
\title{Jump of the Milnor Number in Linear Deformations with Fixed Order}
\begin{document}
	\maketitle
	\begin{abstract}
Let $f_{0}$ be a plane curve singularity. We study the jump of the Milnor number
of a homogeneous singularity $f_{0}$ of order $n$ in the class of linear deformations
$f_{s}=f_{0}+sg$ of $f_0$ with fixed order $\ord f_{s}=j<n$ for $s\neq0$,
where $g$ is holomorphic and $\ord g=j.$
	\end{abstract}
\section{Introduction}

A basic numeric invariant (topological) of an isolated singularity $f_{0}$ is
the Milnor number $\mu(f_{0})$ of $f_{0}.$ In a small, one-parameter holomorphic deformation
$(f_{s})$ of $f_{0}$ the Milnor number decreases and is constant for small
$s\neq0$ by the upper semi-continuity property of $\mu(f_{s})$ in families in the Zariski
topology. It is an interesting phenomenon that for some singularities we do
not always get every number $<\mu(f_{0})$ as the Milnor number $\mu(f_{s})$
for $s\neq0$ of a deformation of $f_{0}.$ For instance, for the singularity
$x^{4}+y^{4}$ in the Arnold $X_{9}$ class, with the Milnor number $9$, there
does not exist its deformation $(f_{s})$ for which $\mu(f_{s})=8,$ $s\neq0$
\cite{BK14}. This phenomenon, in particular the smallest non-zero difference
between $\mu(f_{0})$ and $\mu(f_{s})$, $s\neq 0$ over all deformations of $f_{0}$ (called
the \textbf{jump of the Milnor number}) was studied by S. Gusein-Zade \cite{GZ93},
A. Bodin \cite{Bod07}, Sz. Brzostowski, T. Krasi\'{n}ski and J. Walewska
\cite{Wal13,BK14,KW19,BKW21} and A.
Zakrzewska \cite{Zak17,Zak25} in various classes of singularities and
deformations. The problem of computing the jump of the Milnor number is not
easy because it is not a topological invariant of singularities \cite{BKW21}.

In the article, we consider singularities of plane curves $f_0$ and their \textbf{linear
	deformations} $f_{s}=f_{0}+sg,$ where $g$ is holomorphic and $g(0,0)=0.$ For such
deformations of $f_{0}$ the jump is in fact a topological invariant
\cite{doktorat}. A. Zakrzewska in \cite{Zak25} determined the values of such
jumps for quasihomogeneous singularities (in particular homogeneous). We consider a similar problem: for a
homogeneous plane curve singularity of order $n$ (geometrically, this is a system
of $n$ different lines intersecting at the origin of $\mathbb{C}^{2})$ we determine the
jump of the Milnor number for linear deformations of order $j$ or less,
$2\leqslant j<n$ (it means $\ord g=j)$. The case $j=1$ is
trivial as in this case the jump is equal to $(n-1)^{2}$. We prove that the
jump of $f_0$ for fixed $j$ is equal to $(n-j)(n-2)$ i.e., the maximal generic Milnor number of linear
deformations of $f_{0}$ of order $j$ is equal to $(n-2)j+1$ (Theorem \ref{main}). In particular, for the greatest possible $j=n-1$ we get known result \cite{Zak25} that the jump for homogeneous $f_0$ of order $n$ for linear deformations of $f_0$ is equal to $n-2$ (Corollary \ref{wnio}).

To obtain our results we represent singularities with Enriques diagrams
which completely determine topological types of singularities \cite{CA00}. The
main tool in the proofs is the result by M. Alberich-Carrami\~{n}ana and J.
Ro\'{e} \cite{ACR05}, which gives a necessary and sufficient condition for two
Enriques diagrams of singularities to be linearly adjacent, i.e., when a singularity
is a linear deformation of another.

\section{Enriques diagrams}

Let $f_{0}:(\mathbb{C}^{2},0)\rightarrow(\mathbb{C},0)$ be a singularity i.e.,
$f_{0}$ is the germ of a holomorphic function such that $f_{0}\neq0$,
$f_{0}(0,0)=0$, and $f_{0}$ is reduced in $\mathcal{O}^{2}$. It is equivalent
to the condition that the gradient mapping $\nabla f:(\mathbb{C}%
^{2},0)\rightarrow(\mathbb{C}^{2},0)$ has an isolated zero at $0$. The
zero-set germ at $0$ of $f_{0}$ is denoted by $V(f_{0})$. To represent the
topological type of $f_{0}$, we use the Enriques diagrams \cite{CA00}.
Recall the Enriques diagram of $f_{0}$ is a rooted tree whose vertices are all
infinitely near points of $0$ lying on $V(f_{0})$ in the minimal resolution of $f_{0}$.
It has two types of edges reflecting a natural proximity relation between
vertices, coming from the resolution process. Notice that the Enriques diagram of
$f_{0}$ has no weights, but from it we may read off the orders of the strict
transforms of $f_{0}$ at vertices, the orders of the total transforms of $f_{0}$
at vertices, the Milnor number, the $\delta$-invariant of $f_{0}$, and other
invariants. To define linear adjacency of Enriques diagrams (and thus
singularities), we need a modification of Enriques diagrams (more flexible), introduced in the theorem of M. Alberich-Carrami\~{n}ana and J. Ro\'{e}
--
\textbf{abstract weighted Enriques diagrams}.

\begin{definition}
	[\cite{ACR05}] An abstract weighted Enriques diagram is a pair $(D,\nu),$ where $D$ is a rooted tree
	with a binary relation between vertices, called \textbf{proximity}, and 
	$\nu$ is a function $D\rightarrow\mathbb{Z}_{\geq0}$ which satisfy:
	
	1. the root $R$ is proximate to no vertex,
	
	2. every vertex is proximate to its predecessor,
	
	3. no vertex is proximate to more than two vertices,
	
	4. if a vertex $Q$ is proximate to two vertices, then one of them, according to 2., is the
	immediate predecessor of $Q$ and it is proximate to the other,
	
	5. given two vertices $P,Q$ with $Q$ proximate to $P,$ there is at most one
	vertex proximate to both of them.
\end{definition}

We denote this proximity relation $Q\rightarrow P$, $Q$ is proximate to $P$. The vertices which are
proximate to two other points are called \textbf{satellite} vertices, the other
vertices (except the root) are called \textbf{free} vertices. A vertex is
\textbf{final} (or a \textbf{leaf}) if it has no successor. Abstract weighted
Enriques diagrams are drawn according to the following rules (as in the
classical version):

1. if $Q$ is a free successor of $P,$ then the edge going from $P$ to $Q$ is curved,

2. the sequence of edges connecting a maximal succession of vertices proximate
to the same vertex $P$ is shaped into a line segment, transversal to the edge
joining $P$ to the first vertex of the sequence.

Each singularity $f_{0}$ determines an abstract weighted Enriques diagram
$E(f_{0})$, called the Enriques diagram of $f_0$: the tree is as in the classical version and the weights are the
orders of proper transforms of $f_{0}$ at vertices of the tree. Not every abstract
weighted Enriques diagram $(D,\nu)$ represents a singularity. It represents a singularity if and only if, for all $P\in D$,
\begin{align*}
	\nu(P)  &  >0,\\
	\nu(P)  &  =1\text{ for every final and free }P,\\
	\nu(P)  &  =\sum_{Q\rightarrow P}\nu(Q)\text{ for non-final }P.
\end{align*}
Such abstract weighted Enriques diagrams will be called \textbf{complete}.
However, non-trivial ($\nu(R)>0$) abstract weighted Enriques diagrams which satisfy only the inequality%
\begin{equation}
	\nu(P)\geqslant\sum_{Q\rightarrow P}\nu(Q)\text{ } \label{1}%
\end{equation}
for all $P\in D$ also represent singularities. Namely, we may complete any such
diagram to a unique complete one by adding to vertices $P$ for which $\nu
(P)>\sum_{Q\rightarrow P}\nu(Q)$ edges to new free vertices of weights 1 in
the number $\nu(P)-\sum_{Q\rightarrow P}\nu(Q)$ and removing vertices of the
weight $0.$ Diagrams satisfying (\ref{1}) will be called \textbf{consistent}.
In other words, each consistent abstract weighted Enriques diagram represents a
unique singularity. An example of a consistent diagram and its completion is
shown in Figure \ref{com}. In the sequel two consistent abstract weighted
Enriques diagrams representing the same singularity will be considered as
\textbf{equivalent}. It is easy to see (\cite{Zak17},Theorem 2.12).

\begin{figure}	\centering
	\begin{subfigure}{8em}
		\begin{tikzpicture}
			\node[draw=none] (r) at (0.2,0) {$\bullet$};
			\node[draw=none] (r) at (0.4,0.2) {$_7$};					
			\draw[fill=none](0.2,0.018) circle (0.12);						
			\node[draw=none] (p2) at (0,-1) {$\bullet$};
			\node[draw=none] (p2) at (0.2,-0.8) {$_3$};					
			\node[draw=none] (p3) at (0.8,-1) {$\bullet$};
			\node[draw=none] (p3) at (1,-0.8) {$_2$};													
			\node[draw=none] (p6) at (0.2,-1.7) {$\bullet$};
			\node[draw=none] (p6) at (0.4,-1.6) {$_0$};								
			\node[draw=none] (p7) at (0.8,0.4) {$\bullet$};	
			\node[draw=none] (p7) at (1,0.5) {$_1$};									
			\draw[-] (0.2,0) edge[bend right] (0,-1);		
			\draw[-] (0,-1) to (0.8,-1);		
			\draw[-] (0,-1) edge[bend left] (0.2,-1.7);			
			\draw[-] (0.2,0) edge[bend right] (0.8,0.4);										
		\end{tikzpicture}
		\caption{}
	\end{subfigure}
	\begin{subfigure}{8em}
		\begin{tikzpicture}
			\node[draw=none] (r) at (0.2,0) {$\bullet$};
			\node[draw=none] (r) at (0.4,0.2) {$_7$};					
			\draw[fill=none](0.2,0.018) circle (0.12);						
			\node[draw=none] (p2) at (0,-1) {$\bullet$};
			\node[draw=none] (p2) at (0.2,-0.8) {$_3$};					
			\node[draw=none] (p3) at (0.8,-1) {$\bullet$};
			\node[draw=none] (p3) at (1,-0.8) {$_2$};					
			\node[draw=none] (p4) at (1.3,-1.7) {$\bullet$};
			\node[draw=none] (p4) at (1.5,-1.6) {$_1$};					
			\node[draw=none] (p5) at (0.7,-1.7) {$\bullet$};
			\node[draw=none] (p5) at (0.9,-1.6) {$_1$};								
			\node[draw=none] (p6) at (-0.3,-1.7) {$\bullet$};
			\node[draw=none] (p6) at (-0.1,-1.6) {$_1$};								
			\node[draw=none] (p7) at (0.8,-0.4) {$\bullet$};	
			\node[draw=none] (p7) at (1,-0.3) {$_1$};
			\node[draw=none] (p7) at (0.8,0.4) {$\bullet$};	
			\node[draw=none] (p7) at (1,0.5) {$_1$};									
			\draw[-] (0.2,0) edge[bend right] (0,-1);		
			\draw[-] (0,-1) to (0.8,-1);	
			\draw[-] (0.8,-1) edge[bend left] (1.3,-1.7);
			\draw[-] (0.8,-1) edge[bend left] (0.7,-1.7);	
			\draw[-] (0,-1) edge[bend right] (-0.3,-1.7);
			\draw[-] (0.2,0) edge[bend right] (0.8,-0.4);			
			\draw[-] (0.2,0) edge[bend right] (0.8,0.4);																
		\end{tikzpicture}
		\caption{}
	\end{subfigure}
	\caption{The consistent diagram (a) and its completion (b).\label{com}}	
\end{figure}
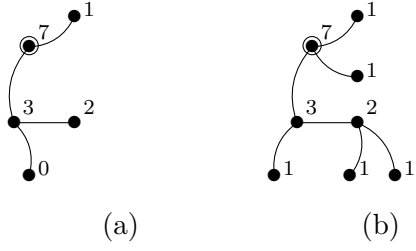

\begin{lemma}
	For each consistent abstract weighted Enriques diagram $(D,\nu)$ there exists
	exactly one minimal consistent abstract weighted Enriques diagram $(D_{\min
	},\nu_{\min})$ which represents the same singularity as $(D,\nu). $
\end{lemma}

The \textbf{minimality} means that $D$ has no vertices of the weight $0$ and
no free vertices of the weight $1$, except those $P\in D$ for which there exists
a satellite vertex $Q\in D$ satisfying $Q\rightarrow P.$ By the lemma and
known theorems, we get

\begin{theorem}
	There exists a bijection between minimal consistent abstract weighted Enriques
	diagrams and topological types of singularities.
\end{theorem}

In the sequel an \textbf{Enriques diagram} means a consistent abstract weighted Enriques diagram. For Enriques diagrams we may define, as in the
classical case, the total order function, the Milnor number, and other
notions. So, for an Enriques diagram $(D,\nu)$ we define

1. the \textbf{total order} $\ord_{\nu}$ at vertices $P\in$
$D$
\begin{equation}
	\ord_{\nu}(P):=\left\{
	\begin{array}
		[c]{ll}%
		\nu(P) & \text{if $P$ is the root,}\\
		\nu(P)+\sum\limits_{P\rightarrow Q}\ord_\nu (Q) & \text{otherwise.}%
	\end{array}
	\right.  \label{ord}%
\end{equation}

2. the \textbf{Milnor number} of $(D,\nu)$
\[
\mu(D,\nu):=\sum_{P\in D}\nu(P)(\nu(P)-1)+1-r_{D},
\]
where
\begin{equation}\label{branch}
	r_{D}:=\sum_{P\in D}\left(  \nu(P)-\sum_{Q\rightarrow P}\nu(Q)\right)  .
\end{equation}

Of course, if $(D,\nu)$ is a complete Enriques diagram representing a
singularity $f_{0}$, then the above notions coincide with classical notions:
the orders of total transforms of $f_{0}$ at vertices of $D$ and the Milnor number
$\mu(f_{0})$ of $f_{0}.$ M. Alberich-Carrami\~{n}ana and J. Ro\'{e} in
\cite{ACR05} Theorem 1.3 and Remark 1.4, gave a necessary and sufficient condition
for two Enriques diagrams of singularities to be linearly adjacent, which means
that there exists a linear deformation for which the first diagram is the
Enriques diagram of a specific member of the deformation and the second diagram
is the Enriques diagram of the generic member.

\begin{theorem}
	\label{CR}For two minimal Enriques diagrams $(D_{0},\nu_{0})$ and $(D,\nu)$ the following conditions are equivalent:
	
	1. For every singularity $f_{0}$ with the Enriques diagram $(D_{0},\nu_{0})$
	there is a linear deformation $(f_{s})_{s\in S}$ of $f_{0}$ such that the
	Enriques diagram of a generic element $f_{s}$ is $(D,\nu)$.
	
	2. There exists a singularity $f_{0}$ with the Enriques diagram $(D_{0}%
	,\nu_{0})$ for which there is a linear deformation $(f_{s})_{s\in S}$ of
	$f_{0}$ such that the Enriques diagram of a generic element $f_{s}$ is
	$(D,\nu)$.
	
	3. There exists an Enriques diagram $(D_{0}^{\prime},\nu_{0}^{\prime})$
	equivalent to $(D_{0},\nu_{0})$ such that $D_{0}^{\prime}$ and $D$ have
	isomorphic (without weights) subdiagrams $\widetilde{D}_{0}^{\prime}\subset
	D_{0}^{\prime}$ and $\widetilde{D}\subset D$ such that for a new weight
	$\tilde{\nu}$ on $D$ induced from $D_{0}^{\prime}$ by this isomorphism (for
	vertices in $D\setminus\widetilde{D}$ we put the weights $0$), the two total
	order functions on $D$ (calculated by formula (\ref{ord})) satisfy
	$$
	\ord _{\tilde{\nu}}(P)\geq\ord _{\nu}(P),P\in D.
	$$
	
\end{theorem}

The condition in item $3$ of the theorem is called \textbf{linear adjacency} of
$(D_{0},\nu_{0})$ to $(D,\nu).$ In general, for arbitrary two Enriques
diagrams $(D_{0},\nu_{0})$ and $(D,\nu)$ we say that $(D_{0}%
,\nu_{0})$ is \textbf{linearly adjacent} to $(D,\nu)$ if and only if their
minimal Enriques diagrams are linearly adjacent. We denote this by $(D_{0}%
,\nu_{0})\rightarrow(D,\nu).$

Since the Milnor number is upper semi-continuous in families of isolated
singularities (\cite{GLS} Theorem 2.6 I and Proposition 2.57 II) we get the
following corollary

\begin{corollary}
	If $(D,\nu)$ and $(D_{0},\nu_{0})$ are Enriques diagrams and
	$(D_{0},\nu_{0})$ is \textbf{linearly adjacent} to $(D,\nu)$ then
	$$
	\mu(D_{0},\nu_{0})\geq\mu(D,\nu).
	$$
	
\end{corollary}

\section{The jump of the Milnor number in linear deformations}

Let $f_{0}$ be a singularity and $(f_{s})_{s\in S}$ its holomorphic
deformation, where $S$ is a neighbourhood of $0\in\mathbb{C}$. Since the Milnor
number $\mu_{s}:=\mu(f_{s},0)$, $s\in S$, is upper semi-continuous in the
Zariski topology we may assume that $\mu_{s}=\text{const.}$ for $s\in
S\setminus\{0\}$ and $\mu_{0}\geq\mu_{s}$ The constant difference $\mu_{0}%
-\mu_{s}$ (for $s\neq0$) will be called \textbf{the jump of the deformation }
$(f_{s})$ and denoted by $\lambda((f_{s}))$. The smallest non-zero value among
the jumps of all deformations of $f_{0}$ will be called \textbf{the jump of the
	Milnor number of the singularity} $f_{0}$ and denoted by $\lambda(f_{0})$.

If we restrict the class of deformations of $f_{0}$ to a specific one
$\mathcal{D}$ this jump will be denoted by $\lambda^{\mathcal{D}}(f_{0})$. In
particular, $\lambda^{lin}(f_{0})$ is the jump (called in the sequel \textbf{linear
	jump} of $f_{0}$) in the class of \textbf{linear deformations} $f_{s}=f_{0}%
+sg$, where $g$ is holomorphic, $g(0,0)=0$. Similarly, $\lambda^{nd}(f_{0})$ is the jump in the
class of \textbf{non-degenerate deformations} (i.e., each $f_{s}$ is
non-degenerate in the Kushnirenko sense). In general, the jump $\lambda(f_0)$ is not
a topological invariant of $f_{0}$ \cite{BKW21}. However, if we restrict
deformations to linear ones, this is true by Theorem \ref{CR}.

\begin{proposition}
	[\cite{doktorat}]If singularities $f_{0}$ and $\tilde{f_{0}}$ have the
	same topological type, then $\lambda^{lin}(f_{0})=\lambda^{lin}(\tilde{f_{0}})$.
\end{proposition}

Results concerning the jumps in various classes of deformations and
singularities one can be found in \cite{GZ93,Bod07,Wal13,KW19,BKW21}. The basic fact, which make this theory interesting, is that there
exist singularities for which $\lambda(f_{0})>1$ \cite{GZ93}. For instance,
for $f_{0}(x,y)=x^{n}+y^{n}$, $n\geq3$, we have $\lambda(f_{0})=\left[
\frac{n}{2}\right]  $ \cite{BKW21}, $\lambda^{lin}(f_{0})=n-2$ \cite{doktorat}, and
$\lambda^{nd}(f_{0})=n-1$ \cite{Wal13}.

\bigskip The main theorem of the article concerns linear jumps of homogeneous
singularities with a fixed order of deformations.

\begin{theorem}
	\label{main}If $f_{0}$ is a homogeneous singularity of order $n\geq3$ the
	jump of $f_{0}$ in the class of linear deformations $f_{s}=f_{0}+sg$ of order $j$, $2\leq j<n$ is equal to $(n-j)(n-2).$
\end{theorem}

\begin{remark}\label{nj}
	\bigskip For $n=2$ or $j=1$ the above problem of finding $\lambda(f_0)$ is trivial. So, in the sequel we will assume that $\ord f\geq 3$. In turn, for $j=n$ we have $\ord(f_0+sg)\geq n$ which implies $\mu(f_0+sg)\geq (n-1)^2$, so there are no "jumps" in such cases.
\end{remark}

From the theorem, taking $j=n-1$, we obtain a known result.

\begin{corollary}
	[\cite{doktorat}]\label{wnio}The linear jump of the Milnor number $\lambda^{lin}(f_{0})$ of a
	homogeneous singularity $f_{0}$ of order $n$ is equal to $(n-2).$
\end{corollary}

\section{Proof of Theorem \ref{main}}

\begin{proof}
	Let $f_{0}$ be a homogeneous singularity of order $n\geq3$ and fix $j$, $2\leq
	j<n.$ Notice $f_{0}$ is a homogeneous polynomial of degree $n$ which is the
	product of linear, not proportional, factors of the type $ax+by$ in the number
	$n.$ The assertion of the theorem is equivalent to the fact that the greatest
	Milnor number of a linear deformation of $f_{0}$ of order $j$ is equal to
	$(n-2)j+1.$ First, we prove that any linear deformation of $f_{0}$ of order $j$
	has the Milnor number (generic) at most $(n-2)j+1.$ Next, we give an
	algorithm for constructing a linear deformation of $f_{0}$ of order $j$
	which generic Milnor number is exactly $(n-2)j+1.$
	
	First, we show some results concerning Enriques diagrams of arbitrary singularities
	and a linear deformation $f_{s}$ of $f_{0}$. Since the linear jump is a
	topological invariant of $f_{0}$ it suffices to consider the 
	Enriques diagrams associated with $f_{0}$ and a generic $f_{s}$. Let
	$(E_{0},\nu_{0})$ (resp. $(E,\nu))$ be the minimal Enriques
	diagram of $f_{0}$ (resp. of a generic element $f_{s}$ of the linear deformation
	$(f_{s})$) and according to Theorem \ref{CR} we have linear adjacency $(E_{0}%
	,\nu_{0})\rightarrow(E,\nu).$ This adjacency gives a new weight function on
	$E$(\footnote{induced from an Enriques diagram $E_0'$ equivalent to $E_0$ and being an extension of $E_0$ by free vertices of weight 1 see Theorem \ref{CR}.\label{foot}}) which we will also denote by $\nu_{0}.$ By the
	definition of linear adjacency we have%
	\[
	\ord _{\nu_{0}}(P)\geq\ord _{\nu}(P),P\in E.
	\]
	We introduce several notions associated to $E$ and its two weight functions
	$\nu_{0}$ and $\nu,$ and two total order functions $\ord %
	_{\nu_{0}}$ and $\ord _{\nu}.$ Recall the value $\nu_{0}(P),$
	$P\in E,$ is defined by the isomorphism described in the definition of linear
	adjacency. If $P$ is in the image of this isomorphism we preserve the weight,
	and if $P$ is outside this image $\nu_{0}(P)=0$. Let $R$ be the root of $E$
	and $l(R,P)$ be the length of the path from $R$ to $P$ in the tree $E.$ Recall
	it is the number of blowing-ups under which we get $P$ ($P$ is in the
	$l(R,P)$-th infinitesimal neighbourhood of the origin)$.$ We define subsets of
	$E$ depending on the $i$-th infinitesimal neighbourhood of the origin for
	$i=0,1,\ldots,k:=$ the depth of the tree $E:$%
	\begin{align*}
		L_{i}  &  :=\{P\in E:l(R,P)=i\},\\
		\tilde{E_i}  &  :=%
		{\displaystyle\bigcup\limits_{j\leq i}}
		L_{i}=\{P\in E:l(R,P)\leq i\},\\
		L_{i}^{\prime}  &  :=L_{i}\cup\{Q\in E:\exists_{P\in L_{i}}P\rightarrow Q\},\\
		A_{i}  &  :=\{Q\in E:\exists_{P\in L_{i}}P\rightarrow Q\text{ and }Q\text{ is
			the predecessor of }P\},\\
		B_{i}  &  :=\{Q\in E:\exists_{P\in L_{i}}P\rightarrow Q\text{ and }Q\text{ is
			not the predecessor of }P\}.
	\end{align*}
	Obviously, $(\tilde{E_i},\nu_{0}|\tilde{E_i})$ and $(\tilde{E_i},\nu|\tilde{E_i})$ are consistent, and
	$L_{i}^{\prime}=L_{i}\cup A_{i}\cup B_{i}.$ Moreover, $A_{i}\cup B_{i}\subset
	L_{i-1}^{\prime},$ $i=1,\ldots,k$. We define also the following functions:%
	\begin{align*}
		\delta(P)  &  :=\ord _{\nu_{0}}(P)-\ord _{\nu
		}(P)\geq0,\text{ \ \ }P\in E,\\
		\varkappa_{i}(P)  &  :=\nu(P)-\sum_{Q\rightarrow P,Q\in\tilde{E_i}}\nu
		(Q)\geq0,\text{ \ \ }P\in \tilde{E_i},\text{\ }i=0,\ldots,k\\
		V(i)  &  :=\sum_{P\in L_{i}^{\prime}}\delta(P)\varkappa_{i}(P)\geq0,\text{
			\ \ }i=0,\ldots,k.
	\end{align*}
	The following recursive inequality for $V(i)$ plays a crucial role in
	the proof.
	
	\begin{lemma}
		For every $i=1,\ldots,k$ we have
		\begin{equation}
			V(i)\leq V(i-1)+\sum_{P\in L_{i}}(\nu_{0}(P)-\nu(P))\nu(P). \label{rec}%
		\end{equation}
		
	\end{lemma}
	
	\begin{proof}
		By the definition of $\delta$ for every $P\in E$ we have: $\delta
		(P)=\ord _{\nu_{0}}(P)-\ord _{\nu}(P)=\sum
		_{P\rightarrow Q}(\ord _{\nu_{0}}(Q)-\ord _{\nu
		}(Q))+\nu_{0}(P)-\nu(P)=\sum_{P\rightarrow Q}\delta(Q)+\nu_{0}(P)-\nu(P)$, and
		for $P\in L_{i}$, we have $\varkappa_{i}(P)=\nu(P),$ $i=0,\ldots,k.$ In turn, for every
		$Q\in L_{i}^{\prime}\setminus L_{i}=A_{i}\cup B_{i}$
		\[
		\varkappa_{i}(Q)=\varkappa_{i-1}(Q)-\sum_{P\rightarrow Q,P\in L_{i}}\nu(P).
		\]
		Since $L_{i}^{\prime}=L_{i}\cup A_{i}\cup B_{i}$ we may compute $V(i)$
		\begin{align*}
			V(i)  &  =\sum_{P\in L_{i}^{\prime}}\delta(P)\varkappa_{i}(P)=\sum_{Q\in
				A_{i}\cup B_{i}}\delta(Q)\varkappa_{i}(Q)+\sum_{P\in L_{i}}\delta
			(P)\varkappa_{i}(P)\\
			&  =\sum_{Q\in A_{i}\cup B_{i}}\delta(Q)(\varkappa_{i-1}(Q)-\sum_{P\rightarrow
				Q,P\in L_{i}}\nu(P))\\
			&  +\sum_{P\in L_{i}}(\sum_{P\rightarrow Q}\delta(Q)+\nu_{0}(P)-\nu
			(P))\nu(P)\\
			&  =\sum_{Q\in A_{i}\cup B_{i}}\delta(Q)\varkappa_{i-1}(Q)-\sum_{Q\in
				A_{i}\cup B_{i}}\delta(Q)\sum_{P\rightarrow Q,P\in L_{i}}\nu(P)\\
			&  +\sum_{P\in L_{i}}\sum_{P\rightarrow Q}\delta(Q)\nu(P)+\sum_{P\in L_{i}%
			}(\nu_{0}(P)-\nu(P))\nu(P)\\
			&  =\sum_{Q\in A_{i}\cup B_{i}}\delta(Q)\varkappa_{i-1}(Q)+\sum_{P\in L_{i}%
			}(\nu_{0}(P)-\nu(P))\nu(P).
		\end{align*}
		Since $A_{i}\cup B_{i}\subset L_{i-1}^{\prime}$, we obtain%
		\begin{align*}
			V(i)  &  \leq\sum_{Q\in L_{i-1}^{\prime}}\delta(Q)\varkappa_{i-1}%
			(Q)+\sum_{P\in L_{i}}(\nu_{0}(P)-\nu(P))\nu(P)\\
			&  =V(i-1)+\sum_{P\in L_{i}}(\nu_{0}(P)-\nu(P))\nu(P).
		\end{align*}
	\end{proof}
	
	Recursive formula (\ref{rec}) and the obvious inequality $V(k)\geq0$ give%
	
	\[
	0\leq V(0)+\sum_{P\neq R}(\nu_{0}(P)-\nu(P))\nu(P).
	\]

	Since $V(0)=\delta(R)\varkappa_{0}(R)=(\nu_{0}(R)-\nu(R))\nu(R)$ we obtain
	from the above inequality%
	
	\begin{equation}
		\sum_{P\neq R}\nu(P)^{2}\leq\sum_{P\neq R}\nu_{0}(P)\nu(P)+\nu_{0}%
		(R)\nu(R)-\nu(R)^{2} \label{ine}%
	\end{equation}

	Using the last inequality we get some inequality for the Milnor number of
	$E.$ First, we rewrite the formula for the Milnor number%
	\begin{align*}
		\mu(E)  &  =\sum_{P\in E}\nu(P)(\nu(P)-1)+1-r_{E}\\
		&  =\sum_{P\in E}\nu(P)(\nu(P)-1)+1-\sum_{P\in E}(\nu(P)-\sum_{Q\rightarrow
			P}\nu(Q))\\
		&  =\sum_{P\in E}\nu(P)^{2}-2\sum_{P\in E}\nu(P)+\sum_{P\in E}\sum
		_{Q\rightarrow P}\nu(Q)+1.
	\end{align*}

	Using (\ref{ine}) and separating free and satellite vertices we get
	\begin{equation}
		\mu(E)\leq\nu(R)(\nu_{0}(R)-2)+\sum_{P\text{-free}}\nu(P)(\nu_{0}%
		(P)-1)+\sum_{P\text{-satellite}}\nu(P)\nu_{0}(P)+1. \label{nier}%
	\end{equation}
	Notice that we have obtained this inequality without assuming that $f_{0}$ is a
	homogeneous singularity. It holds for any pair $(E_{0},\nu_{0})\rightarrow
	(E,\nu).$ If we assume $f_{0}$ is a homogeneous singularity we get from this
	inequality the first part of the theorem, which we formulate as a separate theorem.
	
	\begin{theorem}
		Let $f_{0}$ be a homogeneous singularity of order $n\geq3$ and
		$f_{s}=f_{0}+sg,$ where $g$ is holomorphic and $2\leq\ord g=j<n,$ be any
		linear deformation of $f_{0}$ of order $j.$ Then for small $s\neq0$%
		\[
		\mu(f_{s})\leq j(n-2)+1.
		\]
		
	\end{theorem}
	
	\begin{proof}
		Let $(E_{0},\nu_{0})$ and $(E,\nu)$ be minimal Enriques diagrams of $f_{0}$
		and $f_{s}$ for small $s\neq0,$ respectively. By assumption
		$E_{0}=\{R_{E_0}\},$ $\nu_{0}(R_{E_0})=n,$ $\nu(R_{E_0})=j$ and $(E_{0},\nu_{0})\rightarrow
		(E,\nu).$ Hence from (\ref{nier}) we get
		\[
		\mu(E)\leq j(n-2)+\sum_{P\text{-free}}\nu(P)(\nu_{0}(P)-1)+\sum
		_{P\text{-satellite}}\nu(P)\nu_{0}(P)+1.
		\]
		Since $E_{0}$ consists of only one vertex, linear adjacency of
		$(E_{0},\nu_{0})$ to $(E,\nu)$ (recall we may extend $E_{0}$ by free vertices
		of the weight 1) implies that if $P\in E$ is a satellite vertex, then $\nu_{0}(P)=0$ and if
		$P\in E$ is a free vertex, then $\nu_{0}(P)\leq1$. Hence
		\[
		\sum_{P\text{-free}}\nu(P)(\nu_{0}(P)-1)\leq0,\sum_{P\text{-satellite}}%
		\nu(P)\nu_{0}(P)=0.
		\]
		Therefore,
		\[
		\mu(E)\leq j(n-2)+1.
		\]
	\end{proof}
	
	Now we turn to the second part of the theorem. We formulate it as the following theorem.
	
	\begin{theorem}\label{main2}
		Let $f_{0}$ be a homogeneous singularity of order $n\geq3$ and fix an
		integer $j,$ $2\leq j<n.$ Then there exists a linear deformation$\ f_{s}%
		=f_{0}+sg,$ $g$ holomorphic and order $j$, such that $\mu(f_{s})=j(n-2)+1$
		for small $s\neq0.$
	\end{theorem}
	
	\begin{proof}
		Let $(E_{0},\nu_{0})$ be the minimal Enriques diagram of $f_{0}.$ It consists
		of only one point $E_{0}=\{R_{0}\}$ and $\nu_{0}(R_{0})=n.$ By Theorem
		\ref{CR} it suffices to construct an Enriques diagram $(E,\nu)$ such
		that $(E_{0},\nu_{0})\rightarrow(E,\nu),$ $\nu(R)=j$ where $R$ is the root of
		$E,$ and $\mu(E)=j(n-2)+1.$ We inductively construct a finite sequence of
		Enriques diagrams $(E_{i},\nu_{i}),$ $i=1,\ldots,k$ for some
		$k\geq1$ which the last element $(E_{k},\nu_{k})$ satisfies the required
		conditions: $\nu_{k}(R_{E_k})=j,$ $(E_{0},\nu_{0})\rightarrow(E_{k},\nu_{k})$
		and $\mu(E_{k})=j(n-2)+1.$ Algorithm for $(E_{i},\nu_{i})$ will be divided
		into two cases: the first one when $2j\geq n+1$ and the second one when
		$2j<n+1.$ The latter case will be reduced to the first one.
		
		1. Algorithm in the case $2j\geq n+1.$ In this case we will construct
		$(E_{i},\nu_{i})$, $i\geq 1$, which are minimal and consistent, where each $E_{i}$ is a bamboo which
		all vertices, except the first and the second ones are satellite, $E_{i+1}$ is an
		extension of $E_{i}$ by one vertex and $(E_{0},\nu
		_{0})\rightarrow(E_{i},\nu_{i})$.
		
		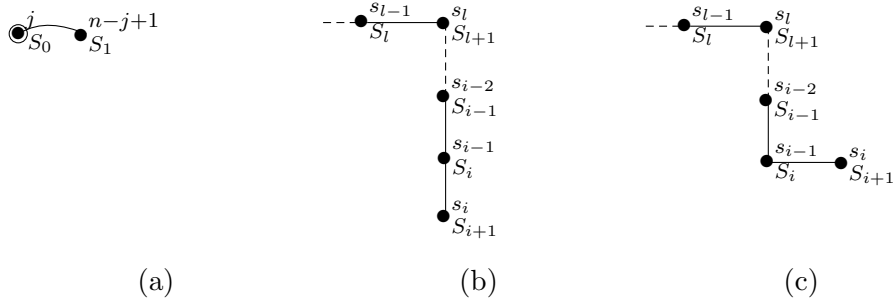
\begin{figure}[h]
			\begin{center}
				\begin{subfigure}{10em}
					\begin{tikzpicture}[scale=1]
						\node (p0) at (0.5,1) {$\bullet$};
						\node (p0) at (0.67,1.13) {$^j$};
						\node (p0) at (0.8,0.88) {$_{S_0}$};						
						\draw (0.5,1) circle (0.12);
						\node (p)  at (1.8,1) {$\bullet_{S_1}^{n-j+1}$};
						\node (w) at (1,-1.7) {};
						\draw[in=160,out=30] (0.5,1) to (1.4,1);			
					\end{tikzpicture}
					\caption{}
				\end{subfigure}
				\begin{subfigure}{10em}
					\begin{tikzpicture}[scale=1]
						\node (r) at (0.7,1) {$\bullet_{S_l}^{s_{l-1}}$};				
						\node (p0) at (1.8,0.95) {$\bullet_{S_{l+1}}^{s_l}$};
						\node (p)  at (1.81,0) {$\bullet_{S_{i-1}}^{s_{i-2}}$};
						\node (u)  at (1.8,-0.8) {$\bullet_{S_i}^{s_{i-1}}$};
						\node (w) at (1.81,-1.6) {$\bullet_{S_{i+1}}^{s_i}$};
						\draw (0.49,1) to (1.52,1);
						\draw[densely dashed] (-0.1,1) to (0.49,1);
						\draw[densely dashed] (1.52,1) to (1.52,0);
						\draw (1.52,0) to (1.52,-0.8);
						\draw (1.52,-0.8) to (1.52,-1.6);
						
					\end{tikzpicture}
					\caption{}
				\end{subfigure}
				\begin{subfigure}{10em}
					\begin{tikzpicture}[scale=1]
						\node (r) at (0.7,1) {$\bullet_{S_l}^{s_{l-1}}$};				
						\node (p0) at (1.81,0.95) {$\bullet_{S_{l+1}}^{s_l}$};
						\node (p)  at (1.81,0) {$\bullet_{S_{i-1}}^{s_{i-2}}$};
						\node (u)  at (1.8,-0.8) {$\bullet_{S_i}^{s_{i-1}}$};
						\node (w) at (2.8,-0.85) {$\bullet_{S_{i+1}}^{s_i}$};
						\node (z) at (2,-1.8) {};
						\draw (0.49,1) to (1.52,1);
						\draw[densely dashed] (-0.1,1) to (0.49,1);
						\draw[densely dashed] (1.52,1) to (1.52,0);
						\draw (1.52,0) to (1.52,-0.8);
						\draw (1.52,-0.8) to (2.5,-0.8);
						
					\end{tikzpicture}
					\caption{}
				\end{subfigure}
				
				\caption{The diagrams: (a) $(E_1,\nu_1)$, (b) $(E_{i+1},\nu_{i+1})$ in Case 1, (c) $(E_{i+1},\nu_{i+1})$ in Case 2 \label{indu}}
				
			\end{center}
		\end{figure}
		
		\textbf{Initialization.} Let $E_{1}=\{S_{0},S_{1}\}$ ($S_0$ is the root), with $S_{1}\rightarrow
		S_{0}$, $\nu_{1}(S_{0})=j$, $\nu_{1}(S_{1})=n-j+1$ (see Figure \ref{indu}(a)). Then $(E_{1},\nu_{1})$
		satisfies all the above conditions.
		
		\textbf{Inductive step.} Assume that we have constructed a minimal and
		consistent Enriques diagram $(E_{i},\nu_{i})$, $i\geq 1$, where $E_{i}=\{S_{0},\ldots,S_{i}\}$,
		$E_{i}$ is a bamboo i.e., $S_{i}\rightarrow S_{i-1}\rightarrow\ldots\rightarrow
		S_{1}\rightarrow S_{0}$, vertices $S_{2},\ldots,S_{i}$ are satellite and
		$(E_{0},\nu_{0})\rightarrow(E_{i},\nu_{i})$. We may assume $S_{i}\rightarrow
		S_{i-1},S_{l}$ (for $i=1$ only $S_{1}\rightarrow S_{0}$) for some $l<i-1.$
		Define the following values, which measure the differences between
		the corresponding weights and orders at the vertices involved in
		the inductive construction:
		\begin{align*}
			a_{i}  &  =\ord _{\nu_{0}}(S_{l})-\ord _{\nu_{i}%
			}(S_{l})\geq0,\\
			b_{i}  &  =\ord _{\nu_{0}}(S_{i-1})-\ord _{\nu
				_{i}}(S_{i-1})\geq0,\\
			c_{i}  &  =\ord _{\nu_{0}}(S_{i})-\ord _{\nu_{i}%
			}(S_{i})\geq0,\\
			d_{i}  &  =\nu_{i}(S_{l})-\sum_{P\rightarrow S_{l}}\nu_{i}(P)\geq0,\\
			e_{i}  &  =\nu_{i}(S_{i-1})-\nu_{i}(S_{i})\geq0,\\
			f_{i}  &  =\nu_{i}(S_{i})>0.
		\end{align*}

		We define $(E_{i+1},\nu_{i+1})$ as an extension of $(E_{i},\nu_{i})$ by
		one vertex in cases depending on values of the above sequences:

		\textbf{Case 1.} $s_{i}:=\min(a_{i}+c_{i},d_{i},f_{i})>0$. We define
		$E_{i+1}=\{S_{0},\ldots,S_{i},S_{i+1}\}.$ The proximity relations and the weights $\nu_{i+1}(S_{j})$ for $S_{0},\ldots,S_{i}$ are the same as in
		$(E_{i},\nu_{i}).$ For $S_{i+1}$ we put $S_{i+1}\rightarrow S_{i},S_{l},$ and
		$\nu_{i+1}(S_{i+1}):=s_{i}$ (see Figure \ref{indu}(b)). The Enriques diagram $(E_{i+1},\nu_{i+1})$ is consistent. Indeed, it suffices to check consistency of
		$(E_{i+1},\nu_{i+1})$ only at $S_{l}$ and $S_{i}.$
		\begin{align*}
			\nu_{i+1}(S_{l})-\sum\limits_{\substack{P\rightarrow S_{l},\\P\in E_{i+1}}%
			}\nu_{i+1}(P)  &  =\nu_{i}(S_{l})-\sum\limits_{\substack{P\rightarrow
					S_{l},\\P\in E_{i}}}\nu_{i}(P)-\nu_{i+1}(S_{i+1})\\
			&  =d_{i}-s_{i}\geq0,
		\end{align*}
		and
		\[
		\nu_{i+1}(S_{i})-\nu_{i+1}(S_{i+1})=\nu_{i}(S_{i})-\nu_{i+1}(S_{i+1}%
		)=f_{i}-s_{i}\geq0.
		\]
		In turn, to check $(E_0,\nu_0)\rightarrow(E_{i+1},\nu_{i+1})$ it suffices to check
		linear adjacency only at
		$S_{i+1}.$
		\begin{align*}
			\ord _{\nu_{0}}(S_{i+1})-\ord _{\nu_{i+1}}%
			(S_{i+1})  &  =\ord _{\nu_{0}}(S_{l})-\ord %
			_{\nu_{i}}(S_{l})+\ord _{\nu_{0}}(S_{i})\\
			-\ord _{\nu_{i}}(S_{i})-\nu_{i+1}(S_{i+1})  &  =a_{i}%
			+c_{i}-s_{i}\geq0.
		\end{align*}
		So, $(E_{i+1},\nu_{i+1})$ satisfies all the required conditions of the algorithm.
		
		\textbf{Case 2.} $\min(a_{i}+c_{i},d_{i},f_{i})=0$ and $s_{i}:=\min(b_{i}%
		+c_{i},e_{i},f_{i})>0$. Similarly as above we define $(E_{i+1},\nu_{i+1})$
		with one modification. In this case, the proximity relations for $S_{i+1}$ is
		$S_{i+1}\rightarrow S_{i},S_{i-1}$ (see Figure \ref{indu}(c)). The Enriques diagram $(E_{i+1},\nu_{i+1})$ is consistent. Indeed, it suffices to check consistency of
		$(E_{i+1},\nu_{i+1})$ only at $S_{i-1}$ and $S_{i}$
		\begin{align*}
			\nu_{i+1}(S_{i-1})-\sum\limits_{\substack{P\rightarrow S_{i-1},\\P\in E_{i+1}%
			}}\nu_{i+1}(P) &  =\\
			\nu_{i}(S_{i-1})-\sum\limits_{\substack{P\rightarrow S_{i-1},\\P\in E_{i}}%
			}\nu_{i}(P)-\nu_{i+1}(S_{i+1}) &  =e_{i}-s_{i}\geq0
		\end{align*}
		and
		\[
		\nu_{i+1}(S_{i})-\nu_{i+1}(S_{i+1})=\nu_{i}(S_{i})-\nu_{i+1}(S_{i+1}%
		)=f_{i}-s_{i}\geq0
		\]
		In turn, to check $(E_0,\nu_0)\rightarrow(E_{i+1},\nu_{i+1})$ it suffices to check linear adjacency $(E_{0},\nu_{0})\rightarrow(E_{i+1},\nu_{i+1})$ only at
		$S_{i+1}.$ We have
		\begin{align*}
			\ord _{\nu_{0}}(S_{i+1})-\ord _{\nu_{i+1}}%
			(S_{i+1}) &  =\ord _{\nu_{0}}(S_{i-1})-\ord %
			_{\nu_{i}}(S_{i-1})\\
			+\ord _{\nu_{0}}(S_{i})-\ord _{\nu_{i}}(S_{i}%
			)-\nu_{i+1}(S_{i+1}) &  =b_{i}+c_{i}-s_{i}\geq0.
		\end{align*}
		So, $(E_{i+1},\nu_{i+1})$ satisfies all the required conditions of the algorithm.
		
		\textbf{Case 3.} $\min(a_{i}+c_{i},d_{i},f_{i})=0$ and $\min(b_{i}+c_{i}%
		,e_{i},f_{i})=0$. This terminates the inductive construction. We define
		the required Enriques diagram by $(E,\nu):=(E_{i},\nu_{i})$ so that $k=i$. We always get this
		case because at each step we strictly decrease the sum $d_i+e_i+f_i$ (it is formula (\ref{end}) below, which is one of properties of the sequences $a_i,\ldots,f_i$). By construction
		we have that $(E,\nu)$ is minimal, consistent, $(E_{0},\nu_{0})\rightarrow
		(E,\nu)$ and $\nu(R)=j,$ where $R=S_0$ is the root of $E.$ So, it suffices to
		prove $\mu(E)=j(n-2)+1.$ To check this equality we have to come back to the
		function $V(i)$ defined in the first part of the proof. $V(i)$ was defined for
		any minimal Enriques diagram $(E,\nu)$ such that $(E_{0},\nu
		_{0})\rightarrow(E,\nu).$ Since in our case $E=\{S_{0},\ldots,S_{k}\}$ and it
		is a bamboo then the depth of the tree $E$ is $k.$ Hence by the definitions of
		$V(i)$ and the above sequences $a_{i},b_{i},c_{i},d_{i},e_{i},f_{i}$ we get
		\[
		V(i)=\left\{
		\begin{matrix}
			(n-j)j & \text{ for }i=0\\
			a_{1}d_{1}+c_{1}f_{1} & \text{ for }i=1\\
			a_{i}d_{i}+b_{i}e_{i}+c_{i}f_{i} & \text{ for }i=2,\ldots,k
		\end{matrix}
		\right.
		\]
		We will prove that
		\begin{align}
			V(i+1) &  =V(i)-s_{i}^{2},\text{ \ \ \ \ }i=2,\ldots,k-1\label{45}\\
			V(k) &  =0\label{46}%
		\end{align}

		First, we notice that six sequences $a_{i},b_{i},c_{i},d_{i},e_{i},f_{i}$
		can be defined recursively according to the following cases:
		
		1. Initial values are as follows: $a_{1}=b_{1}=n-j>0$, $c_{1}=0$, $d_{1}%
		=e_{1}=2j-n-1>0$, $f_{1}=n+1-j>0.$
		
		2. In \textbf{Case 1} i.e., $s_{i}:=\min(a_{i}+c_{i},d_{i},f_{i})>0$ we get from
		definitions $a_{i+1}=a_{i}$, $b_{i+1}=c_{i}$, $c_{i+1}=a_{i}+c_{i}-s_{i}$,
		$d_{i+1}=d_{i}-s_{i}$, $e_{i+1}=f_{i}-s_{i}$, $f_{i+1}=s_{i},$
		
		3. In \textbf{Case 2} i.e., $\min(a_{i}+c_{i},d_{i},f_{i})=0$ and $s_{i}:=\min
		(b_{i}+c_{i},e_{i},f_{i})>0$ we get from definitions $a_{i+1}=b_{i}$,
		$b_{i+1}=c_{i}$, $c_{i+1}=b_{i}+c_{i}-s_{i}$, $d_{i+1}=e_{i}-s_{i}$,
		$e_{i+1}=f_{i}-s_{i}$, $f_{i+1}=s_{i},$
		
		These recurrences imply, by induction, the following properties of these sequences:
		
		1. In \textbf{Case 1}
		\begin{equation}
			b_{i}+c_{i}=0,a_{i}\leq f_{i}\text{ \ or \ }a_{i}+e_{i}=0,c_{i}\leq
			d_{i}, i=1,\ldots,k.\label{47}%
		\end{equation}

		2. In \textbf{Case 2}
		\begin{equation}
			b_{i}+d_{i}=0,c_{i}\leq e_{i}\text{ \ or \ }a_{i}+c_{i}=0,b_{i}\leq
			f_{i}, i=1,\ldots,k.\label{48}%
		\end{equation}
		
		3. 	\begin{equation}
			1\leq d_{i+1}+e_{i+1}+f_{i+1}<d_i+e_i+f_i, i=1,\ldots,k-1.\label{end}%
		\end{equation}
		
		The above properties imply (\ref{46}). In fact, for $i=k$ we are in Case 3
		i.e., $\min(a_{k}+c_{k},d_{k},f_{k})=0$ and $\min(b_{k}+c_{k},e_{k},f_{k})=0.$
		Since we have always $f_{i}>0$ then $\min(a_{k}+c_{k},d_{k})=0$ and
		$\min(b_{k}+c_{k},e_{k})=0.$ Hence either $c_{k}=0$ or $d_{k}=e_{k}=0$. If
		$c_{k}=0$ then we have $\min(a_{k},d_{k})=0$ and $\min(b_{k},e_{k})=0.$ This
		implies $V(k)=a_{k}d_{k}+b_{k}e_{k}+c_{k}f_{k}=0.$ In turn, if $d_{k}=e_{k}=0$ then:
		
		$\bullet$ if previous case was \textbf{Case 1} and the first part of alternative (\ref{47})
		holds i.e., $b_{k}+c_{k}=0$, then we have $c_{k}=0$,
		
		$\bullet$ if previous case was \textbf{Case 1} and the second part of alternative (\ref{47})
		holds i.e., $a_{k}+e_{k}=0$ and $c_{k}\leq d_{k}=0$, then we have $c_{k}=0$,
		
		$\bullet$ if previous case was \textbf{Case 2} and the first part of alternative (\ref{48})
		holds i.e $b_{k}+d_{k}=0$ and $c_{k}\leq e_{k}=0$, then we have $c_{k}=0$,
		
		$\bullet$ if previous case was \textbf{Case 2} and the second part of alternative (\ref{48})
		holds i.e $a_{k}+c_{k}=0$, then we have $c_{k}=0$.
		
		\noindent In all subcases we have $c_{k}=0$. Hence, $V(k)=a_{k}d_{k}+b_{k}e_{k}%
		+c_{k}f_{k}=0$. 
		
		Now, we will show (\ref{45}).
		
		$\bullet$  In \textbf{Case 1}: $V(i+1)=a_{i+1}d_{i+1}+b_{i+1}e_{i+1}+c_{i+1}f_{i+1}%
		=a_{i}(d_{i}-s_{i})+c_{i}(f_{i}-s_{i})+(a_{i}+c_{i}-s_{i})s_{i}=a_{i}%
		d_{i}-a_{i}s_{i}+c_{i}f_{i}-c_{i}s_{i}+a_{i}s_{i}+c_{i}s_{i}-s_{i}^{2}%
		=a_{i}d_{i}+c_{i}f_{i}-s_{i}^{2}$. If $i=1$ then $V(2)=V(1)-s_{1}^{2}$. If
		$i>1$, since it is \textbf{Case 1} from (\ref{47}) we have $b_{i}+c_{i}=0$ or $a_{i}+e_{i}=0$. Then
		$b_{i}e_{i}=0$, so $V(i+1)=a_{i}d_{i}+c_{i}f_{i}+b_{i}e_{i}-s_{i}^{2}%
		=V(i)-s_{i}^{2}$.
		
		$\bullet$  In \textbf{Case 2}: $V(i+1)=a_{i+1}d_{i+1}+b_{i+1}e_{i+1}+c_{i+1}f_{i+1}%
		=b_{i}(e_{i}-s_{i})+c_{i}(f_{i}-s_{i})+(b_{i}+c_{i}-s_{i})s_{i}=b_{i}%
		e_{i}-b_{i}s_{i}+c_{i}f_{i}-c_{i}s_{i}+b_{i}s_{i}+c_{i}s_{i}-s_{i}^{2}%
		=b_{i}e_{i}+c_{i}f_{i}-s_{i}^{2}$. Since it is Case 2 (\ref{48}) we have $b_{i}+d_{i}=0$
		or $a_{i}+c_{i}=0$. Then $a_{i}d_{i}=0$, so $V(i+1)=b_{i}e_{i}+c_{i}f_{i}%
		+a_{i}d_{i}-s_{i}^{2}=V(i)-s_{i}^{2}$.
		
		\noindent This ends the proof of (\ref{45}).
		
		Thus from (\ref{45}) and (\ref{46})
		\[
		0=V(k)=V(k-1)-s_{k-1}^{2}=\ldots=V(1)-s_{1}^{2}-\ldots-s_{k-1}%
		^{2},
		\]
		so
		\[
		s_{1}^{2}+\ldots+s_{k-1}^{2}=V(1)=(n-j)(2j-n-1).
		\]

		By (\ref{branch}) the number of branches in $E$ is $r_{E}=j-s_{1}-\ldots-s_{k-1}$. Then
		\begin{align*}
			\mu(E)  &  =j(j-1)+(n-j+1)(n-j)\\
			&  +s_{1}(s_{1}-1)+\ldots+s_{k-1}(s_{k-1}-1)+1-r_{E}\\
			&  =j^{2}-j+n^{2}-nj-nj+j^{2}+n-j\\
			&  +s_{1}^{2}-s_{1}+\ldots+s_{k-1}^{2}-s_{k-1}+1-j+s_{1}%
			+\ldots+s_{k-1}\\
			&  =2j^{2}-3j+n^{2}-2nj+n+1+s_{1}^{2}+\ldots+s_{k-1}^{2}\\
			&  =2j^{2}-3j+n^{2}-2nj+n+1+(n-j)(2j-n-1)\\
			&  =2j^{2}-3j+n^{2}-2nj+n+1+3nj-n^{2}-n-2j^{2}+j\\
			&  =-2j+nj+1=j(n-2)+1.
		\end{align*}

		2. Algorithm in the case $2j<n+1$. We reduce this case to the first one taking as the initial part of construction the sequence (a bamboo) of free vertices of weight $j$ as many as possible i.e., until we get the assumptions of the first case are fulfilled. Precisely, let $t\in\mathbb{N}$ be such that $n+t>(t+1)j$ and $n+t+1\leq(t+2)j$. Define $n'=n+t-tj$ and $j'=j$. Then $2j'\geq n'+1$. So, taking the Enriques diagram $(E_0',\nu_0')$ of homogeneous singularity of order $n'$ and value $j'$, we can apply the first case. Let $(E',\nu')$, $E'=\{S_0,\ldots,S_k\}$, be the diagram obtained in the first case for $n',j'$. Then we define $(E,\nu)$ as: $E=\{P_1,\ldots,P_t\}\cup E'$, where $P_1,\ldots,P_t$ are new vertices, $P_1$ is the root and $P_{i+1}\to P_i$, $\nu(P_i)=j$, $S_0\to P_t$ and $\nu_{|E'}=\nu'$.
		
		The diagram $(E,\nu)$ is consistent since $P_1,\ldots,P_t$ are free, $\nu(P_{1}%
		)=\ldots=\nu(P_{t})=\nu(S_{0})=j$ and $(E^{\prime},\nu^{\prime})$ is consistent.
		
		Moreover, for $i=1,\ldots,t$ $\ord _{\nu_{0}}(P_{i})=n+i-1$ (see footnote (\textsuperscript{\ref{foot}}) on page \pageref{foot}),
		$\ord _{\nu}(P_{i})=ij$, $\ord _{\nu_{0}}%
		(P_{i})-\ord _{\nu}(P_{i})=n+i-1-ij=n-1-i(j-1)\geq
		n-1-t(j-1)\geq j-1>0$. It is easy to see that, for $S\in E'$, 
		$\ord _{\nu_{0}}(S)-\ord _{\nu}%
		(S)=\ord _{\nu_{0}'%
		}(S)-\ord _{\nu'}(S)\geq0$. Thus $(E_{0},\nu_{0}%
		)\to(E,\nu)$.
		
		Finally, $r_{E}=r_{E^{\prime}}$, and then
		\begin{align*}
			\mu(E)  &  =\mu(E^{\prime})+tj(j-1)\\
			&  =(n+t-tj-2)j+1+tj(j-1)\\
			&  =nj+tj-tj^{2}-2j+1+tj^{2}-tj\\
			&  =nj-2j+1=j(n-2)+1.
		\end{align*}
	\end{proof}
	\renewcommand{\qedsymbol}{}
\end{proof}

\section{Concluding remarks}
\noindent 1. The algorithmic construction given in the proof of Theorem \ref{main2}, which delivers a linear deformation with the maximal generic Milnor number, is not unique. For instance, for the same homogeneous singularity $(E_0,\nu_0)$ of order $n\geq 3$ and a fixed $j$, $2\leq j<n$, when $j$ is even i.e., $j=2k$ we may give another construction. Putting $l:=n-j$ we define $E=\{R,P_1^1,\ldots,P_1^k,\ldots,\allowbreak P_l^1,\ldots,P_l^k\}$, with weights $\nu(P^m_i)=2$, $\nu(R)=j$, and proximity relations $P^m_1\to R$ and $P_i^m\to P_{i-1}^m$. We easily check that $(E_0,\nu_0)\to(E,\nu)$ and $\mu(E,\nu)=j(n-2)+1$.

\bigskip
\noindent 2. By Theorem \ref{main} the greatest Milnor number of a linear deformation of a
singularity $f_{0}$ (homogeneous and order $n$) that we may attain is $(n-2)(n-1)+1$. It gives rise to a question if all less numbers $<(n-2)(n-1)+1$ are also generic Milnor numbers of linear deformations of $f_0$. It is an interesting question because in other classes of deformations it is not true. For example, there are some extraordinary gaps in the sequence of generic Milnor numbers in the class of non-degenerate deformations of $f_0$ \cite{BKW21}. In our considered case (linear deformations) it could not happen. All numbers $\mu<(n-2)(n-1)+1$ are generic Milnor numbers of linear deformations of $f_0$ (in the forthcoming article). It is a particular case of a more general problem whether for given $n,j$, $2\leq j<n$, all numbers $\mu<(n-2)j+1$ are generic Milnor numbers of linear deformations of $f_0$ of order $\leq j$. It is also solved in the forthcoming article. In fact, it is true with one exception: if $n$ is odd and $j=3$, the value $\mu=3n-6$ (provided it is $>1$) is not realizable.

\bigskip
\noindent 3. The main theorem concerns homogeneous singularities. Since the linear jump is a topological invariant, the results also hold for semi-homogeneous singularities i.e., singularities of the form $f_0=\tilde{f_0}+\tilde{\tilde{f_0}}$, where $\tilde{f_0}$ is a homogeneous isolated singularity of order $n$ and $\ord\tilde{\tilde{f_0}}>n$. The Enriques diagrams of $f_0$ and $\tilde{f_0}$ are the same.

\bigskip
\noindent 4. Some of the results (in particular, Lemma \ref{rec} and Formula (\ref{46})) hold for any pair of adjacent singularities $(E_0,\nu_0)\to(E,\nu)$, not necessarily with $(E_0,\nu_0)$ homogeneous. This suggests that the main result may be extended to a wider class of singularities (for instance quasi-homogeneous).

\bigskip
\noindent 5. Our original goal was to characterize all singularities $f_0$ for which $\lambda^{lin}(f_0)$ is equal to 1. While carrying out this task, we encountered the problem solved in this article. The mentioned characterization will be the subject of a future article.

\bigskip
\noindent 6. The result of this paper can be also applied to the following interesting problem of singularities: for a given $f_0$ characterize all linear $\delta$-constant deformations of $f_0$. In such deformations, according to the Milnor formula $\mu=2\delta-r+1$, the jump (downward) of the Milnor number must be equal to the jump (upward) of the number of branches. This will the subject of a future article.
	
\bibliographystyle{alpha}
\bibliography{biblio}	

@article{GZ93,
    author  = "Sabir Gusein-Zade",
    title   = "On singularities from which an {$A_1$} can be split off.",
    year    = "1993",
    journal = "Funct. Anal. Appl.",    
    volume  = "27",
    number  = "1",
    pages   = "57--60"
}

@book{GLS,
    author    = "Gert-Martin Greuel and Christoph Lossen and Eugenii Shustin",
    title     = "Introduction to {S}ingularities and {D}eformations",
    year      = "2006",
    publisher = "Springer Verlag",
	address="Berlin"
}

@book{CA00,
    author    = "Eduardo Casas-Alvero",
    title     = "Singularities of {P}lane {C}urve",
    year      = "2000",
    publisher = "Cambridge University Press",
	address="Cambridge "
}

@article{ACR05,
    author  = "Maria  Alberich-Carramiñana and Joaquim Roé",
    title   = "Enriques diagrams and adjacency of planar curve singularities",
    year    = "2005",
    journal = "Can. J. Math.",
    volume  = "57",
    number  = "1",
    pages   = "3--16"
}

@article{BKW21,
    author  = "Szymon Brzostowski  and Tadeusz Krasi\'nski and Justyna Walewska",
    title   = "Milnor numbers in deformations of homogeneous singularities",
    year    = "2021",
    journal = "Bull. Sci. Math.",
    volume  = "168",
    number   = "102973"
}

@article{Wal13,
    author  = "Justyna Walewska",
    title   = "The jump of {M}ilnor numbers in families of non-degenerate and non-convenient singularities",
    year    = "2013",
    journal = "Analytic and Algebraic Geometry, University of Łódź Press",
    pages   = "141--153"
}

@article{Bod07,
    author  = "Arnaud Bodin",
    title   = "Jump of {M}ilnor numbers",
    year    = "2007",
    journal = "Bull. Braz. Math. Soc. (N.S.)",
    volume  = "38",
    number  = "3",
    pages   = "389--396"
}

@article{Zak17,
    author  = "Aleksandra Zakrzewska",
    title   = "The jump of {M}ilnor number for linear deformations of homogeneous singularities",
    year    = "2017",
    journal = "Bull. Soc. Sci. Lett. Łódź, Sér. Rech. Déform.",
    volume  = "67",
	number = "3/2017",
    pages   = "77--88"
}

@article {BK14,
      author = "Szymon Brzostowski and Tadeusz Krasi\'nski",
      title = "The jump of the {M}ilnor number in the {$X_9$} singularity class",
      journal = "Cent. Eur. J. Math.",
      year = "2014",
      volume = "12",
      number = "3",
      pages="429--435",
}

@article {KW19,
      author = "Tadeusz Krasi\'nski and Justyna Walewska",
      title = "Non-degenerate jumps of the {M}ilnor numbers of quasihomogeneous singularities",
      journal = "Ann. Pol. Math.",
      year = "2019",
      volume = "123",
      pages="369--386",
}

@phdthesis {doktorat,
      author = "Aleksandra Zakrzewska",
      title = "The jump of {M}ilnor number for linear deformations of plane curve singularities",
		school  = "University of Lodz",
	  year = "2019",
      note = "http://hdl.handle.net/11089/31100 (PhD thesis in Polish)"
}

@article{Zak25,
      title={The jump of the {M}ilnor number of quasihomogeneous singularities for linear deformations}, 
      author={Aleksandra Zakrzewska},
      year={2025},
      journal="J. Singul.",
	  volume="28",
	  pages="23--38"
}
\begin{flushright}\footnotesize
	\textsc{Tadeusz Krasiński\\
		Aleksandra Zakrzewska\\
		Faculty of Mathematics and Computer Science\\ University of Lodz\\ul. Banacha 22\\90-238 Lodz, Poland\\aleksandra.zakrzewska@wmii.uni.lodz.pl}
\end{flushright}
\end{document}